\documentclass[reqno,centertags, 12pt]{amsart}
\usepackage{amsmath,amsthm,amscd,amssymb}
\usepackage{esint}  
\usepackage{latexsym}
\usepackage{graphicx}
\usepackage{enumitem}
\usepackage[mathscr]{eucal}
\newcommand{\bbC}{{\mathbb{C}}}
\newcommand{\bbD}{{\mathbb{D}}}

\newcommand{\bbR}{{\mathbb{R}}}

\newcommand{\fre}{{\mathfrak{e}}}
\newcommand{\frf}{{\frak{f}}}
\newcommand{\frg}{{\frak{g}}}

\newcommand{\calC}{{\mathcal{C}}}
\newcommand{\calE}{{\mathcal{E}}}

\newcommand{\calG}{{\mathcal G}}
\newcommand{\calH}{{\mathcal H}}

\newcommand{\lb}{\label}
\newcommand{\f}{\frac}

\newcommand{\Arg}{\text{\rm{Arg}}}

\newcommand{\supp}{\text{\rm{supp}}}

\newcommand{\bi}{\bibitem}

\newcommand{\beq}{\begin{equation}}
\newcommand{\eeq}{\end{equation}}
\newcommand{\ba}{\begin{align}}
\newcommand{\ea}{\end{align}}

\newcommand{\cint}{\varointctrclockwise}

\newcounter{smalllist}

\newcommand{\comm}[1]{}

\allowdisplaybreaks
\numberwithin{equation}{section}

\newtheorem{theorem}{Theorem}[section]

\newtheorem*{p2.1}{Proposition 2.1}
\newtheorem{proposition}[theorem]{Proposition}
\newtheorem{lemma}[theorem]{Lemma}
\newtheorem{corollary}[theorem]{Corollary}
\theoremstyle{definition}
\newtheorem{example}[theorem]{Example}

\newtheorem*{remark}{Remark}
\newtheorem*{remarks}{Remarks}

\newcommand{\norm}[1]{\lVert#1\rVert}

\begin{document}

\title[Chebyshev Polynomials, III]{Asymptotics of Chebyshev Polynomials,\\ III. Sets Saturating Szeg\H{o}, Schiefermayr, and Totik--Widom Bounds}
\author[J.~S.~Christiansen, B.~Simon and M.~Zinchenko]{Jacob S.~Christiansen$^{1,4}$, Barry Simon$^{2,5}$ \\and
Maxim~Zinchenko$^{3,6}$}

\thanks{$^1$ Centre for Mathematical Sciences, Lund University, Box 118, 22100 Lund, Sweden.
 E-mail: stordal@maths.lth.se}

\thanks{$^2$ Departments of Mathematics and Physics, Mathematics 253-37, California Institute of Technology, Pasadena, CA 91125.
E-mail: bsimon@caltech.edu}

\thanks{$^3$ Department of Mathematics and Statistics, University of New Mexico,
Albuquerque, NM 87131, USA; E-mail: maxim@math.unm.edu}

\thanks{$^4$ Research supported in part by Project Grant DFF-4181-00502 from the Danish Council for Independent Research.}

\thanks{$^5$ Research supported in part by NSF grants DMS-1265592 and DMS-1665526 and in part by Israeli BSF Grant No. 2014337.}

\thanks{$^6$ Research supported in part by Simons Foundation grant CGM--281971.}

\

\date{\today}
\keywords{Chebyshev polynomials, Szeg\H{o} lower bound, Schiefermayr lower bound, Totik--Widom upper bound}
\subjclass[2010]{41A50, 30C80, 30C10}

\begin{abstract}  We determine which sets saturate the Szeg\H{o} and Schiefermayr lower bounds on the norms of Chebyshev Polynomials.
We also discuss sets that saturate the Totik--Widom upper bound.
\end{abstract}

\maketitle

\section{Introduction} \lb{s1}

Let $\fre \subset \bbC$ be a compact, not finite set.  For any continuous, complex--valued function, $f$, on $\fre$, let
\begin{equation}\label{1.1}
  \norm{f}_\fre = \sup_{z \in \fre} |f(z)|
\end{equation}
The Chebyshev polynomial, $T_n$, of $\fre$ is the (it turns out unique) degree $n$ monic polynomial that minimizes $\norm{P}_\fre$ over all degree $n$ monic polynomials, $P$.  We define
\begin{equation}\label{1.2}
  t_n = \norm{T_n}_\fre
\end{equation}
This paper continues our study \cite{CSZ1, CSYZ2} of $t_n$ and $T_n$, especially their asymptotics as $n \to \infty$.  We let $C(\fre)$ denote the logarithmic capacity of $\fre$ (see \cite[Section 3.6]{HA} or \cite{AG, Helms, Landkof, MF, Ransford} for the basics of potential theory).

Szeg\H{o} \cite{SzLB} proved for all compact $\fre \subset \bbC$ and all $n$ that
\begin{equation}\label{1.3}
  t_n \ge C(\fre)^n
\end{equation}
while Schiefermayr \cite{SchLB} proved if $\fre \subset \bbR$, then
\begin{equation}\label{1.4}
  t_n \ge 2 C(\fre)^n
\end{equation}

This paper had its genesis in a question asked us by J.~P.~Solovej about which $\fre$ have equality in \eqref{1.3} or \eqref{1.4}.  After we found the solution described below, we found that for $\fre \subset \bbR$ the question was answered by Totik \cite{Totik2011} using, in part, ideas of Peherstorfer \cite{Per} (related ideas appear earlier in Sodin--Yuditskii \cite{SY2}).  Moreover, for a special set of domains in $\bbC$, it was answered implicitly (without proof) in Totik \cite{Totik2012}.  We feel it appropriate to publish our proofs because \cite{Totik2012} is neither explicit nor comprehensive and mainly because our proofs are different and, we feel, illuminating.  In addition, the sets $\frf_n$ which we introduce in Section \ref{s3} may be useful in the future.  Here are our two main results:

\begin{theorem} [Totik \cite{Totik2011}] \lb{T1.1} Let $\fre \subset \bbR$.  Fix $n$.  Then $t_n = 2C(\fre)^n$ if and only if there is a polynomial $P$ of degree $n$ so that
\begin{equation}\label{1.5}
   \fre = P^{-1}([-2,2])
\end{equation}
\end{theorem}

\begin{remarks} 1. We emphasize that in \eqref{1.5}, we mean that any $z \in \bbC$ with $P(z) \in [-2,2]$ has $z \in \fre$ (as well as $P(\fre)=[-2,2]$) not just for $z \in \bbR$.

2.  It is easy to see that $T_n$ is then a multiple of $P$.

3.  In particular, if $t_1 = 2 C(\fre)$ and $\fre \subset \bbR$, then $\fre$ is an interval and equality holds in \eqref{1.4} for all $n$.  We note that Totik \cite[Theorem 3]{Totik2014} has a stronger related result.  He proves that if $\lim_{n\to\infty}\norm{T_n}_\fre/C(\fre)^n = 2$ for some $\fre\in\bbR$, then $\fre$ is an interval.

4.  Totik mentions that the ideas in the result and proof are mainly in Peherstorfer \cite{Per}.  While Schiefermayr was Peherstorfer's student and \eqref{1.4} was
in Schiefermayr's thesis, \cite{SchLB} was published 7 years after \cite{Per}.  The sets for which equality holds in \eqref{1.4} are precisely the sets that Peherstorfer
called $T$-sets and which Sodin--Yuditskii \cite{SY2} call $n$-regular sets.  They are precisely the spectra of the period $n$ Jacobi matrices which we called period-$n$ sets in \cite{CSZ1}.
\end{remarks}

\begin{theorem} \lb{T1.2} Let $\fre \subset \bbC$.  Fix $n$.  Then $t_n = C(\fre)^n$ if and only if there is a polynomial, P, of degree $n$ with
\begin{equation}\label{1.6}
  O\partial(\fre) = P^{-1}(\partial\bbD)
\end{equation}
where $O\partial$ is the outer boundary.
\end{theorem}

\begin{remarks} 1.  If $\fre$ is compact, then $\bbC\setminus\fre$ has exactly one unbounded component, $\fre^\sharp$.  Its boundary is $O\partial(\fre)$.  We call $\bbC\setminus\overline{\fre^\sharp}$ the interior of $O\partial\fre$ and $\fre^\sharp$ the exterior of $O\partial\fre$.

2.  We'll state several equivalent forms of this theorem in Section \ref{s3} below.

3.  When $\fre$ is a finite union of analytic Jordan curves lying exterior to each other, this result is stated in passing and without proof in Totik \cite{Totik2012}.  In that case, $O\partial\fre = \fre$ so Totik doesn't mention outer boundaries.

4. Polynomial inverse images of $\partial\bbD$ are called lemniscates (see \cite{lem}).  We'll say more about their structure in Section \ref{s3}, but we note that generically they are a union of at most $\deg(P)$ disjoint mutually exterior analytic Jordan curves and in general, a union of at most $\deg(P)$ piecewise analytic Jordan curves with disjoint interiors but with possible intersections at finitely many points.

5. It is easy to see that $T_n$ is a multiple of $P$.

6.  In particular, $t_1 = C(\fre)$ if and only if $O\partial\fre$ is a circle.
\end{remarks}

It follows from these theorems that if $t_n$ has equality in \eqref{1.3} (resp. $\fre \subset \bbR$ and $t_n$ has equality in \eqref{1.4}), then for any $k = 1,2,\dots$, $t_{nk}$ also has equality in \eqref{1.3} (resp. \eqref{1.4}) (by using a suitable scaling of $P^k$).  We want to note that this can be proven directly:

\begin{theorem} \lb{T1.3} If $t_n$ has equality in \eqref{1.3}, then so does $t_{nk}$ for $k=1,2,\dots$.
\end{theorem}

\begin{proof}  Since $(T_n)^k$ is monic, $t_{nk} = \norm{T_{nk}}_\fre \le \norm{(T_n)^k}_\fre = t_n^k = C(\fre)^{nk}$ if $t_n$ has equality in \eqref{1.3}.  By Szeg\H{o}'s lower bound, we see that $t_{nk} = C(\fre)^{nk}$.
\end{proof}

\begin{theorem} \lb{T1.4} If $\fre \subset\bbR$ and $t_n$ has equality in \eqref{1.4}, then so does $t_{nk}$ for $k=1,2,\dots$.
\end{theorem}

\begin{proof} We can't use $(T_n)^k$ since that only leads to $t_{nk} \le 2^k C(\fre)^{nk}$.  The key is to realize that $z \mapsto z^k$ is the $k$th Chebyshev polynomial for $\{z\,|\, |z| \le t_n\}$, so we replace $z \mapsto z^k$ by the $k$th Chebyshev polynomial $S_k$, for $\frg_n \equiv [-t_n,t_n]$.  Since $C([-t_n,t_n]) = t_n/2$ and equality in \eqref{1.4} holds for all $n$ for intervals, we have that $\norm{S_k}_{\frg_n}=2(t_n/2)^k$.  Since $S_k\circ T_n$ is a monic polynomial of degree $kn$, we have that
\begin{equation*}
  t_{nk} \le \norm{S_k\circ T_n}_\fre \le \norm{S_k}_{\frg_n} = 2(2C(\fre)^n/2)^k = 2C(\fre)^{kn}
\end{equation*}
so, as in the last proof, $t_{nk}=2C(\fre)^{kn}$.
\end{proof}

We prove Theorem \ref{T1.1} in Section \ref{s2}, Theorem \ref{T1.2} in Section \ref{s3}, consider when the upper bound we found in \cite{CSZ1} is optimal in
Section \ref{s3A} and discuss a related problem in Section \ref{s4}.  JSC and MZ would like to thank Fiona Harrison and Elena Mantovan for the hospitality of Caltech where much of this work was done.  We are delighted to dedicate this paper to the memory of Boris Pavlov.  One of us (BS) in particular owes Boris a tremendous debt for having sent him talented undergraduates that Boris mentored in St. Petersburg (Kiselev) and Aukland (Killip) who then did doctoral studies at Caltech.

\section{The Real Case} \lb{s2}

In this section, we'll prove Theorem \ref{T1.1}.  Both it and Theorem \ref{T1.2} rely on the following simple fact.

\begin{proposition} \lb{P2.1}  Let $\fre \subset \frg$ be two compact subsets of $\bbC$ with positive capacity and let $\rho_\frg$ (resp. $\rho_\fre$) be
the potential theoretic equilibrium measure for $\frg$ (resp. $\fre$). Then $C(\fre) = C(\frg)$ if and only if $\supp(\rho_\frg) \subset \fre$.
\end{proposition}

\begin{remark}  Section \ref{s3A} has another proof of this; see Proposition \ref{P3A.1}.
\end{remark}

\begin{proof} Let $\calE(\mu)$ be the logarithmic potential energy of a finite positive measure, i.e.
\begin{equation}\label{2.1}
  \calE(\mu) = \iint \log(|x-y|^{-1}) \, d\mu(x) d\mu(y)
\end{equation}
so that $\rho_\frg$ is the unique probability measure minimizing $\calE(\mu)$ among all probability measures with $\supp(\mu) \subset \frg$.
Since $C(\fre) = e^{-\calE(\rho_\fre)}$, we have that
\begin{equation} \lb{2.1A}
  C(\fre) = C(\frg) \iff \calE(\rho_\fre)=\calE(\rho_\frg) \iff \rho_\fre=\rho_\frg
\end{equation}
for, since $\rho_\fre$ is a trial measure for the $\frg$ potential minimum problem and the minimizer is unique, we have that $\calE(\rho_\fre) \ge \calE(\rho_\frg)$ with equality if and only if $\rho_\fre=\rho_\frg$.

If $\rho_\fre = \rho_\frg$, since $\supp(\rho_\fre) \subset \fre$, we see that $\supp(\rho_\frg) \subset \fre$.  Conversely, if $\supp(\rho_\frg) \subset \fre$,
then $\rho_\frg$ is a trial measure for the $\fre$ potential problem and so the minimizer since it is the minimizer for the larger minimization problem.  It follows that $\rho_\fre=\rho_\frg$ so, by \eqref{2.1A}, $C(\fre) = C(\frg)$.
\end{proof}

Recall that, given $\fre \subset \bbR$, in \cite{CSZ1}, we defined
\begin{equation}\label{2.2}
  \fre_n = T_n^{-1}([-t_n,t_n])
\end{equation}
and proved that
\begin{equation}\label{2.3}
  \fre \subset \fre_n \subset \bbR
\end{equation}
and that
\begin{equation}\label{2.4}
  t_n = 2 C(\fre_n)^n
\end{equation}
It is also easy to see \cite[(2.9)]{CSZ1} that if
\begin{equation}\label{2.5}
  \Delta_n(z) = \frac{2T_n(z)}{t_n}
\end{equation}
then the potential theoretic Green's function for $\fre_n$ is given by
\begin{equation}\label{2.6}
  G_{\fre_n}(z) = \frac{1}{n}\log\left|\frac{\Delta_n(z)}{2}+\sqrt{\left(\frac{\Delta_n(z)}{2}\right)^2-1}\right|
\end{equation}
This can be shown to imply that the equilibrium measure for $\fre_n$ is (\cite[Thm. 2.3]{CSZ1})
\begin{equation}\label{2.7}
  d\rho_{\fre_n}(x) = \frac{|\Delta_n'(x)|}{\,\pi n \sqrt{4-\Delta_n(x)^2}\,}\chi_{\fre_n}(x)\,dx
\end{equation}
where $\chi_{\fre_n}$ is the characteristic function of $\fre_n$.  Since $\Delta_n$ is a polynomial, $\Delta_n'$ is non-vanishing on $\fre_n$ except for a possible finite set in $\fre_n$ (which one can specify precisely but we don't need to).  It follows that

\begin{lemma} \lb{L2.2}
\begin{equation}\label{2.8}
   \supp (\rho_{\fre_n}) = \fre_n
\end{equation}
\end{lemma}

\begin{proof} [Proof of Theorem \ref{T1.1}]  Since $\fre \subset \fre_n$, by Proposition \ref{P2.1}, we have that
\begin{equation} \lb{2.8A}
   C(\fre) = C(\fre_n) \iff \fre_n = \supp(\rho_{\fre_n}) \subset \fre \iff \fre = \fre_n
\end{equation}

On the one hand, by \eqref{2.4}, $t_n=2C(\fre)^n \Rightarrow C(\fre)=C(\fre_n) \Rightarrow \fre=\fre_n \Rightarrow \fre=\Delta_n^{-1}([-2,2])$ so \eqref{1.5} holds with $P=\Delta_n$.  On the other hand, if \eqref{1.5} holds, it is easy to see that $T_n=cP$ and then that $\fre_n = \fre$, so by \eqref{2.4}, we get equality in \eqref{1.4}.
\end{proof}

The above proof is only a slight variant of the proof in Totik \cite{Totik2011}.  We include it mainly to set the stage for the next section.

\section{The Complex Case} \lb{s3}

In this section, we will prove Theorem \ref{T1.2}.  The key to the proof is to define a complex analog of the sets $\fre_n$.  We believe that these sets, $\frf_n$, will be useful elsewhere and are the most important idea in this paper.  Given a compact set $\fre \subset \bbC$ and its Chebyshev polynomial, $T_n$, we define
\begin{equation}\label{3.1}
  \frf_n = \{z\,|\,|T_n(z)| \le t_n\} = T_n^{-1}\bigl(\{z\,|\,|z| \le t_n\}\bigr)
\end{equation}

\begin{theorem} \lb{T3.1} (a)
\begin{equation}\label{3.2}
  \fre \subset \frf_n
\end{equation}
(b)
\begin{equation}\label{3.3}
  \norm{T_n}_\fre = t_n = C(\frf_n)^n
\end{equation}
\end{theorem}

\begin{remarks} 1.  These are analogs of \eqref{2.3} and \eqref{2.4}.

2.  They immediately imply \eqref{1.3} (not that Szeg\H{o}'s proof \cite[Theorem 4.3.7]{OT} is very hard) since \eqref{3.2}$\Rightarrow C(\frf_n) \ge C(\fre)$.
\end{remarks}

\begin{proof} (a) is trivial.

(b) Let $h$ be defined on $\bbC$ by
\begin{equation}\label{3.4}
  h(z) = \left\{
           \begin{array}{ll}
             0, & \hbox{ if } |T_n(z)| \le t_n \\
             \frac{1}{n}\log\left(\frac{|T_n(z)|}{t_n}\right), & \hbox{ if } |T_n(z)| \ge t_n
           \end{array}
         \right.
\end{equation}
Then $h$ is continuous on $\bbC$ and harmonic on $\bbC\setminus\frf_n$ and near infinity has the asymptotics
\begin{equation}\label{3.5}
  h(z) = \log |z| -\tfrac{1}{n} \log(t_n)+\mbox{o}(1)
\end{equation}
From the first term and $h(z)=0$ on $\frf_n$, we see that $h$ is the Green's function, $G_{\frf_n}$, for $\frf_n$.  By the realization of the capacity
in the asymptotics of the Green's function \cite[(3.7.4)\,\&\,(3.7.6)]{HA} and \eqref{3.5}, we see that
\begin{equation*}
  C(\frf_n) = t_n^{1/n}
\end{equation*}
which is \eqref{3.3}.
\end{proof}

The proof of (b) just depended on the form of $\frf_n$ and not that, apriori, $T_n$ is a Chebyshev polynomial.  We thus can prove:

\begin{theorem} \lb{T3.2} Let $P$ be a degree $n$ polynomial with
\begin{equation}\label{3.5A}
  P(z) = cz^n+\dots
\end{equation}
and let
\begin{equation}\label{3.6}
  S_\alpha = \{z\,|\,|P(z)| \le \alpha \}
\end{equation}
for some $\alpha > 0$.  Then
\begin{equation}\label{3.5B}
  C(S_\alpha) = (\alpha/c)^{1/n}
\end{equation}
and for $S_\alpha$, we have $T_n=c^{-1}P$.  In particular, $S_\alpha$ obeys 
\begin{equation}
 \norm{T_n}_{S_\alpha} = C(S_\alpha)^n
\end{equation} 
\end{theorem}

\begin{proof} As in the proof of Theorem \ref{T3.1}, outside of $S_\alpha$, the Green's function is $\tfrac{1}{n}\log(|P(z)|/\alpha)$, whose asymptotics at infinity is $\log(|z|)+\tfrac{1}{n}\log(c/\alpha)+\mbox{o}(1)$ so \eqref{3.5B} holds.

Note that $Q=c^{-1}P$ is a monic polynomial with $\norm{Q}_{S_\alpha}=C(S_\alpha)^n$. By Szeg\H{o}'s lower bound, $\norm{Q}_{S_\alpha} \le t_n$ which implies that $Q=T_n$ by the minimum and uniqueness properties of $T_n$.  \end{proof}

Clearly
\begin{equation}\label{3.6A}
  \partial S_\alpha = \{z\,|\, |P(z)| = \alpha\} \equiv L_\alpha
\end{equation}
This is a \emph{lemniscate} \cite{lem}.   $|P|$ is $C^1$ away from the zeros of $P$ and, using the Cauchy--Riemann equations, it is easy to see that
if $P(z_0) \ne 0$ then $\nabla |P|(z_0) = 0 \Leftrightarrow P'(z_0) = 0$. Hence the critical values of $|P|$ are precisely those $\alpha$ for which there
is a $z_0$ with $P'(z_0)=0$ and $|P(z_0)|=\alpha$.  At non-critical values, $L_\alpha$ is thus a union of disjoint, mutually exterior, analytic Jordan curves.
For $\alpha$ small, the number of curves is exactly the number of distinct zeros of $P$.  As $\alpha$ increases, the number of components changes exactly
as $\alpha$ reaches a critical value, $\alpha_0$, at which point the number of components decreases by the number of critical points (counting multiplicity)
on $L_{\alpha_0}$.  At such values, the closure of the components of the non-critical points are piecewise analytic Jordan curves with disjoint interiors and with corners at the critical points.  For $\alpha$ large, $L_\alpha$ is a single analytic Jordan curve.

We call $S_\alpha$, which is the union of the insides of the Jordan curves in $L_\alpha$, a solid lemniscate.  It is easy to describe the equilibrium measure of such sets.

\begin{theorem} \lb{T3.3} Fix a degree $n$ polynomial $P$ and $\alpha > 0$.  Then

(a)
\begin{equation}\label{3.7}
  d\rho \equiv \frac{1}{2\pi in}\frac{P'(z)}{P(z)} dz \restriction L_\alpha
\end{equation}
is a probability measure

(b) On $L_{\alpha}$, we have that
\begin{equation}\label{3.8}
  \frac{P'(z)}{P(z)}\, dz = \left|\frac{P'(z)}{P(z)}\right|\, |dz|
\end{equation}

(c)
\begin{equation}\label{3.9}
  d\rho = \frac{1}{2\pi n} \frac{d}{|dz|} \Arg(P(z)) |dz| \restriction L_\alpha
\end{equation}

(d) The measure in \eqref{3.7} is the equilibrium measure of $S_\alpha$.

(e) 
\,$\supp(d\rho) = L_\alpha$.
\end{theorem}

\begin{remarks}  1.  The symbol $dz$ on a curve needs an orientation.  We'll specify this orientation in the proof.  Basically, it is counter-clockwise around $S_\alpha$.

2.  The proof shows that each Jordan curve in $L_\alpha$ has $\rho$ measure $k/n$, where $k$ is the number of zeros of $P$ (counting multiplicity) inside that curve.

3.  One can also prove the critical (d) by using the formula for the Green's function and by evaluating the normal derivative of $\log(|P|)$ on $L_\alpha$.
\end{remarks}

\begin{proof}  (a), (b), (c) Since $P$ has no zeros on $L_\alpha$, we can locally define an analytic function $W(z) = \log(P(z))$ on each Jordan curve in $L_\alpha$.  Its derivative is $P'(z)/P(z)$ irrespective of which blanch of $\log$ that we take.  Moreover, if locally $P(z) = \alpha e^{i\theta(z)}$ on each such curve and if we parameterize the curve by arc length, $\gamma(s)$, with the curve oriented so that $S_\alpha$ is to the curve's left, then, for $z_1$ and $z_0$ nearby points with $z_j=\gamma(s_j)$ where $s_1>s_0$, we have that $\theta(s_1) > \theta(s_0)$. This is easy to see using the Cauchy--Riemann equations for $\log(P(z))$ and the fact that its real part increases in the direction outwards from $S_\alpha$. Moreover, if $d\theta(\gamma(s))/ds$ vanishes at $s=s_0$, since $\mbox{Re} \, W(s)$ is constant on $\gamma$ we conclude that $W'(z_0) = 0 \Rightarrow P'(z_0) = 0$.  Thus $d\theta/ds$ is strictly positive except at the critical points which implies that $\theta$ is strictly increasing on $\gamma$.

Clearly,
\begin{equation*}
  \int_{z_0}^{z_1} \frac{P'(z)}{P(z)}\, dz = \log\left(\frac{P(z_1)}{P(z_0)}\right) = \log\left(\frac{\alpha e^{i\theta_1}}{\alpha e^{i\theta_0}}\right) = i(\theta_1-\theta_0)
\end{equation*}
proving that the measure in \eqref{3.7} is a positive measure.  By the argument principle, $n$ times the integral over $L_\alpha$ is the number of zeros in $S_\alpha$, so, the measure has total mass 1.  This proves (a) and the formula for $P'/P$ in terms of $\theta'$ proves (c).  The positivity of the measure in (a) proves (b).

(d) Fix $w \in \bbC\setminus S_\alpha$.  Let $\Gamma$ be a single Jordan curve in $L_\alpha$ and $R$ its interior.  Then $\log(z-w)$ is analytic in a neighborhood of $\overline{R}$, so, by the residue calculus and the definition of $d\mu$, if
\begin{equation}\label{3.10}
  P(z) = c \prod_{j=1}^{n} (z-\zeta_j)
\end{equation}
then
\begin{equation*}
  \int_\Gamma \log(z-w) d\mu(z) = \frac{1}{n} \sum_{\zeta_j \in R} \log(\zeta_j-w)
\end{equation*}
Taking real parts and summing over the Jordan curves, we get
\begin{equation}\label{3.11}
   \int \log|z-w|\, d\mu(z) = \frac{1}{n} \log(|P(w)|/c)
\end{equation}
which we have seen is the Green's function up to a constant.  This implies that $d\mu$ is the equilibrium measure.

(e) We've seen that $\theta'$ is positive except on the finite set of critical points so the support is all of $L_\alpha$.
\end{proof}

The last preliminary we need is

\begin{lemma} \lb{L3.4} Fix $\alpha > 0$ and let $\fre \subset S_\alpha$.  Then
\begin{equation}\label{3.12}
  C(\fre) = C(S_\alpha) \iff L_\alpha \subset \fre
\end{equation}
\end{lemma}

\begin{proof}  Immediate from Proposition \ref{P2.1} and the last theorem.
\end{proof}

\begin{proof} [Proof of Theorem \ref{T1.2}] Suppose equality holds in \eqref{1.3}.  Then $C(\fre) = C(\frf_n)$.  Let $P=T_n/t_n$ so that $\frf_n = S_{\alpha=1}$.
By \eqref{3.12}, $P^{-1}(\partial\bbD) \subset \fre \subset P^{-1}(\bbD)$.  By the second inclusion, $\bbC\setminus P^{-1}(\bbD)$ is contained in the unbounded
component of $\bbC\setminus\fre$.  By the first inclusion, we conclude that $O\partial\fre = P^{-1}(\partial\bbD)$.

Conversely, by Theorem \ref{T3.2}, if \eqref{1.6} holds, let $S_1$ be the solid lemniscate associated to $P$.  By \eqref{1.6} and the lemma, $C(\fre)=C(S_1)$.  By Theorem \ref{T3.2}, the monic multiple, $Q$, of $P$ is the Chebyshev polynomial for $S_1$ and $\norm{Q}_{S_1} = C(S_1)^n$.   Since $\bbC\setminus S_1 \subset \bbC\setminus O\partial\fre$, we have that $\fre \subset S_1$ and thus $\norm{Q}_\fre \le \norm{Q}_{S_1} = C(\fre)^n$.  This implies that $Q$ is the Chebyshev polynomial of $\fre$ and that equality holds in \eqref{1.3}.
\end{proof}

We end this section by exploring some alternate forms and consequences of Theorem \ref{T1.2}.

\begin{corollary} \lb{C3.5}  Let $\fre$ be a compact subset of $\bbC$ so that $\bbC\setminus\fre$ is connected.  Fix $n$. Then $t_n = C(\fre)^n$ if and only if $\fre$ is a solid lemniscate.
\end{corollary}

\begin{remark}  It is fairly easy to prove Theorem \ref{T1.2} from this result.
\end{remark}

\begin{proof} By Theorem \ref{T1.2}, this is equivalent to showing that if $\bbC\setminus\fre$ is connected and $O\partial\fre=L_\alpha$, then $\fre = S_\alpha$.  To say that $O\partial\fre=L_\alpha$ means that the unbounded component of $\bbC\setminus\fre$ is $\bbC\setminus S_\alpha$.  If that is so and there is only one component, then $\bbC\setminus S_\alpha = \bbC\setminus\fre$ so $\fre=S_\alpha$.
\end{proof}

Here are other equivalences that are easy to check given our earlier arguments.

\begin{theorem} \lb{T3.6} $t_n = C(\fre)^n \iff \partial\frf_n \subset \fre$.
\end{theorem}

\begin{theorem} \lb{T3.7} $t_n = C(\fre)^n$ if and only if there is a polynomial, $P$, and $\alpha>0$ so that $L_\alpha \subset\fre\subset S_\alpha$.
\end{theorem}

\section{Equality in a Totik--Widom Upper Bound} \lb{s3A}

In \cite{CSZ1}, we dubbed an upper bound of the form $\norm{T_n}_\fre \le QC(\fre)^n$ a Totik--Widom bound after Widom \cite{Widom} and Totik \cite{Totik09} who proved it when $\fre \subset\bbR$ is a finite gap set.  In that paper, we proved that
\begin{equation}\label{3A.1}
  \norm{T_n}_\fre \le 2\exp(PW(\fre)) C(\fre)^n
\end{equation}
where $PW(\fre) = \sum_{w\in\calC} G_\fre(w)$ with $\calC$ the set of critical points (in $\bbC$) of $G_\fre$ (when $\fre\in\bbR$, they lie in $\bbR$).  PW stands for Parreau--Widom who singled out sets with $PW(\fre) < \infty$ in \cite{Parreau, Widom71}.  We'll call sets that are regular for potential theory and obey this condition, PW sets.  Our main goal in this section is to discuss when one has equality in this bound.

Since we want to say something about a formula for $\frf_n$, we recall the proof in a more general context, beginning with

\begin{proposition} \lb{P3A.1} Let $\fre \subset \frg$ be two compact subsets of $\bbC$ with positive capacity and let $\rho_\frg$ (resp. $\rho_\fre$) be the potential theoretic equilibrium measure for $\frg$ (resp. $\fre$).  Then
\begin{equation}\label{3A.2}
  \log\left(\frac{C(\frg)}{C(\fre)}\right) = \int G_\fre(z) d\rho_\frg(z)
\end{equation}
\end{proposition}

\begin{remark}  Since $G_\fre(z) \ge 0$, this implies that $C(\frg) = C(\fre)$ if and only if $G_\fre(z) = 0$ for $\rho_\frg$-a.e. $z$ in $\supp(\rho_\frg)$.  Since $G_\fre(z)=0 \Rightarrow z\in\fre$ and $G_\fre(z)=0$ for q.e. $z \in \fre$, this happens if and only if $\supp(\rho_\frg) \subset\fre$.  This gives an alternate proof of Proposition \ref{P2.1}
\end{remark}

\begin{proof} It is well-known \cite[Theorem 3.6.8]{HA} that near $z=\infty$, we have that $G_\frf(z)= \log|z|-\log(C(\frf))+\mbox{O}(1/z)$.  Let $h(z) \equiv G_\fre(z)-G_\frg(z)$ and note that
\begin{equation}\label{3A.3}
  h(z)=\log\left(\frac{C(\frg)}{C(\fre)}\right)+\mbox{O}(1/z)
\end{equation}
near $\infty$. Thus $h$ is harmonic on $\bbC\setminus\frg$ and bounded near infinity, so harmonic there.  It is known \cite[Corollary 3.6.28]{HA} that $d\rho_\frg$ is not just the equilibrium measure but it is harmonic measure at $\infty$ in the sense that if $H(z)$ is harmonic and bounded on $(\bbC\cup\{\infty\})\setminus\frg$ with q.e. boundary values on $\partial\frg$, then
\begin{equation}\label{3A.4}
  H(\infty)= \int H(z) d\rho_\frg(z)
\end{equation}
Taking $H=h$ and noting that q.e., $h\restriction \frg = G_\fre$, we get \eqref{3A.2} from \eqref{3A.3}
\end{proof}

\begin{theorem} \lb{T3A.2} \emph{(a)} For any compact $\fre \subset\bbR$,
\begin{equation}\label{3A.5}
  \norm{T_n}_\fre = 2 C(\fre)^n \exp\left(n\int G_\fre(x) \, d\rho_{\fre_n}(x)\right)
\end{equation}
\emph{(b)} For any compact $\frf \subset \bbC$
\begin{equation}\label{3A.6}
  \norm{T_n}_\frf = C(\frf)^n \exp\left(n\int G_\frf(z) \, d\rho_{\frf_n}(z)\right)
\end{equation}
\end{theorem}

\begin{remark} (a) is from \cite{CSZ1}; (b) is new although the proof closely follows the proof of (a) in \cite{CSZ1}.
\end{remark}

\begin{proof}  Immediate from \eqref{2.4}, \eqref{3.3} and \eqref{3A.2}.
\end{proof}

The following restates the proof of \eqref{3A.1} from \cite{CSZ1} and answers the question of when equality holds.

\begin{theorem} \lb{T3A.3}  \eqref{3A.1} holds and if for some $\fre\subset\bbR$ and $n$, we have equality in \eqref{3A.1}, the $\fre$ is an interval.
\end{theorem}

\begin{proof} $\fre_n\setminus\fre$ consists of some number of intervals in the gaps of $\fre$, at most one per gap \cite[Theorem 2.4]{CSZ1} and $\rho_{\fre_n}$ is a purely a.c. measure \cite[Theorem 2.3]{CSZ1}.  In each gap, $K$, there is a single critical point, $w_K$ of $G_\fre$ and these are all the critical points.  Moreover, in each gap, $G_\fre$ is strictly concave so $G_\fre$ takes its maximum value for the gap exactly at the single point $w_K$.  Moreover, $\rho_{\fre_n}(\fre_n\cap K) \le 1/n$ \cite[Theorem 2.4]{CSZ1}, so $\int_K G_\fre(x) \, d\rho_{\fre_n} < G_\fre(w_K) /n$ since $d\rho_{\fre_n}$ is absolutely continuous.  \eqref{3A.1} follows by summing over gaps and we only get equality in \eqref{3A.1} if there are no gaps in $\fre$, i.e. if $\fre$ is a closed interval.
\end{proof}

We can also answer when equality in the upper or lower bound occurs asymptotically along a subsequence.  In our paper with Yuditskii \cite{CSYZ2}, we focused
on subsequences $\{n_j\}_{j=1}^\infty$ where the zeros of $T_{n_j}$ in gaps had limits.  There is at most one zero in each gap, $K$ \cite[Theorem 2.3]{CSZ1}.
Let $\calG$ denote the set of all gaps of $\fre$, i.e. bounded components of $\bbR\setminus\fre$.   In \cite{CSYZ2}, we defined what we called a gap collection,
a subset $\calG_0 \subset \calG$ and for each $K\in\calG_0$, a point $x_K\in K$.  We considered subsequences, $T_{n_j}$, so that for $K\in\calG\setminus\calG_0$, as $n_j \to\infty$, either $T_{n_j}$ has no zero in $K$ or the zero goes to the one of the two edges of $K$ and so that for $K\in\calG_0$, there is a zero for large $n_j$ which goes to $x_K$ as $n_j \to\infty$.  This describes all possible limit points of the set of zeros.

\begin{theorem} \lb{T3A.4} Fix $\fre\subset\bbR$, a compact set obeying the PW condition, and a subsequence with an associated limiting zero gap collection, $\calG_0$ and $\{x_K\}_{K\in\calG_0}$.  Then
\begin{equation}\label{3A.6A}
  \lim_{j \to \infty} \norm{T_{n_j}}_{\fre}/C(\fre)^{n_j}=2 \exp\left(\sum_{K\in\calG_0} G_\fre(x_K)\right)
\end{equation}
\end{theorem}

\begin{proof} For any $K\in\calG$ and any $j$, define
\begin{equation}\label{3A.7}
  v_j(K) = n_j\int_K G_\fre(x) \, d\rho_{\fre_{n_j}}(x)
\end{equation}
and define
\begin{equation}\label{3A.8}
  V(K) = \sup_{x\in K} G_\fre(x) = G_\fre(w_K)
\end{equation}

Since, by the PW hypothesis, $V(K)$ is summable and $v_j(K) \le V(K)$, the dominated convergence theorem implies that
\begin{equation}\label{3A.9}
  \lim_{j\to\infty} \sum_{K\in\calG} v_j(K) = \sum_{K\in\calG} \lim_{j\to\infty} v_j(K)
\end{equation}

If $K\in\calG\setminus\calG_0$, since $\rho_{\fre_{n_j}}(K) \le 1/n_j$ \cite[Theorem 2.4]{CSZ1} and $G_\fre\to 0$ at the edges, $v_j(K)\to 0$.

If $K\in\calG_0$, by \cite[Theorem 5.1]{CSZ1}, there is for $j$ large a single, exponentially small band of $\fre_{n_j}$ entirely in $K$ with $x_K$ in the band
and $\rho_{\fre_{n_j}}(K)=1/n_j$.  It follows that $v_j(K) \to G_\fre(x_K)$.  Thus, by \eqref{3A.9}, $\sum_{K\in\calG} v_j(K) \to \sum_{K\in\calG_0} G_\fre(x_K)$.  By \eqref{3A.5}, we get \eqref{3A.6A}.
\end{proof}

\begin{corollary} \lb{C3A.5}  Fix $\fre\subset\bbR$, a compact set obeying the PW condition, and a subsequence with an associated limiting zero gap collection, $\calG_0$ and $\{x_K\}_{K\in\calG_0}$.  Then

\emph{(a)} If $\calG_0$ is empty, we have
\begin{equation}\label{3A.10}
  \lim_{j \to \infty} \norm{T_{n_j}}_{\fre}/C(\fre)^{n_j} = 2
\end{equation}

\emph{(b)} If $\calG_0=\calG$ and, for each $K$, $x_K=w_K$, the critical point in the gap, we have
\begin{equation}\label{3A.11}
  \lim_{j \to \infty} \norm{T_{n_j}}_{\fre}/C(\fre)^{n_j}=2\exp(PW(\fre))
\end{equation}
\end{corollary}

In general, we cannot say when there exist any subsequences of the type in the Corollary but can with a few extra assumptions (see the discussion after the example). We can analyze an especially simple case completely:

\begin{example} \lb{E3A.6} Fix $0 < a <b$ and let $\fre = [-b,-a]\cup [a,b]$, a two band set symmetric about $0$.  Then for $n$ odd, $T_n$ is odd (by uniqueness
of the Chebyshev polynomial), so the unique zero in the gap $(-a,a)$ is at $x=0$ which, by symmetry, is the critical point of $G_\fre$ in the gap.  Thus the ratio along the odds
is given by \eqref{3A.11}.

On the other hand, for $n$ even, $T_n$ is even, so by simplicity of zeros, non-vanishing at $0$.  Since there is at most one zero in $(-a,a)$, there cannot be any,
so $\calG_0$ is empty and thus, the ratio along the evens is given by \eqref{3A.10}.  In fact, more is true. If
\begin{equation*}
  P(x) = 2-\frac{4(x-b)^2}{(a-b)^2}
\end{equation*}
then $\fre = P^{-1}([-2,2])$, so $\norm{T_{2k}}_\fre = 2 C(\fre)^{2k}$ for all $k$ and the lower bound is an equality for all even numbers.  \qed
\end{example}

In \cite{CSYZ2}, we discussed limits of $T_n/\norm{T_n}_\fre$ for $\fre\subset\bbR$ under a stronger condition than PW called DCT.  If $\fre$ has what we
called a canonical generator, which holds in a generic sense, then \cite[Theorem 5.1]{CSYZ2} every Blaschke product occurs as a limit point of the normalized
Chebyshev polynomials which means one has a limit with any set of simple zeros in any set of gaps.  It follows that in this generic DCT case, the set of
limit points of $\norm{T_n}_\fre/C(\fre)^n$ is exactly the interval $[2,2\exp(PW(\fre))]$.

Finite gap sets are always DCT and it is not hard to see that they have a canonical generator in the sense of \cite{CSYZ2} if and only if the harmonic measures of the bands are rationally independent (except for the trivial relation that they sum to 1).  Moreover, it is known (Totik \cite{Totik2001}) that for sets with $q$ gaps (which is a $2q+2$ dimensional space described by $a_1<b_1<\dots<a_{q+1}<b_{q+1}$) the condition of rationally independent harmonic measures is satisfied on the compliment of a set of dimension $q+2$ so this rational independence condition is highly generic.  We thus have

\begin{theorem}  \lb{T3A.7}  Let $\fre\subset\bbR$ be a set with $q$ gaps so that the harmonic measures of any $q$ of the $q+1$ bands are rationally independent.  Then the set of limit points of $\norm{T_n}_\fre/C(\fre)^n$ is exactly the interval $[2,2\exp(PW(\fre))]$.
\end{theorem}

\section{Invariance of Widom Factors Under Polynomial Preimages} \lb{s4}

This final section is connected to the earlier ones, in that it involves polynomial inverse images, but is otherwise unrelated.  In the work of Widom \cite{Widom} on asymptotics of Chebyshev polynomials, a key object is $\norm{T_n}_\fre/C(\fre)^n$, which we, following Goncharov--Hatino\v{g}lu \cite{GonHat}, call Widom factors.  We want to prove:

\begin{theorem} \lb{T4.1} Let $\fre\subset\bbC$ be a compact set, $P(z)$ a monic polynomial of degree $k\ge 1$, and $\fre_P=P^{-1}(\fre)=\{z\in\bbC \,|\, P(z)\in \fre\}$. Then for every Chebyshev polynomial $T_n$ of $\fre$, the polynomial $T_n\circ P$ is a Chebyshev polynomial of $\fre_P$ and
\begin{equation} \lb{4.1}
   \frac{\norm{T_n}_\fre}{C(\fre)^n} = \frac{\norm{T_n\circ P}_{\fre_P}}{C(\fre_P)^{nk}}.
\end{equation}
\end{theorem}

\begin{lemma} \lb{L4.1A} Let $\fre\subset\bbC$ be a compact set, $p$ a polynomial of degree $k\ge1$ with leading coefficient $1/\gamma$, and $\fre_p$ as above. Then $C(\fre_p)^k=|\gamma|C(\fre)$.
\end{lemma}

\begin{proof} Let $G_\fre$ and $G_{\fre_p}$ be the Green's functions for $\fre$ and $\fre_p$, respectively. Then $G_{\fre_p}= \tfrac{1}{k} (G_\fre\circ p)$ since
both functions are harmonic on $\bbC\setminus\fre_p$, zero q.e. on $\partial\fre_p$, and asymptotically $\log|z|$ at infinity. Comparing the constant terms in the asymptotics at infinity yields the claimed result.  \end{proof}

Suppose $p(z)$, $q(z)$ are two polynomials with $k=\deg(p)\ge 1$. The average of $q$ over $p$ is defined by
\begin{equation} \lb{4.2}
   \sigma_{q|p}(z) = \frac{1}{k} \sum_{\{\zeta\,|\,p(\zeta)=p(z)\}}q(\zeta),
\end{equation}
where the values of $\zeta$ are repeated according to their multiplicity.

\begin{lemma}[\cite{OPZ96}] \lb{L4.2}
The average of $q$ over $p$ is a polynomial in $p$, in fact, $\sigma_{q|p}=\hat{q}\circ p$ for some polynomial $\hat{q}$ of degree at most $\deg(q)/\deg(p)$.
\end{lemma}

\begin{proof} Fix $z\in\bbC$. Then for all sufficiently large $R>0$, by the residue calculus,
\begin{align}
  \sigma_{q|p}(z) &= \frac{1}{2\pi i\deg(p)}\cint_{|\zeta|=R}\frac{q(\zeta)p'(\zeta)}{p(\zeta)-p(z)}\,d\zeta  \nonumber \\
                  &= \sum_{j=0}^\infty \frac{p(z)^j}{2\pi i\deg(p)} \cint_{|\zeta|=R}\f{q(\zeta)p'(\zeta)d\zeta}{p(\zeta)^{j+1}}\,d\zeta \lb{4.3}
\end{align}
by picking $R$ so large that $|\zeta|=R \Rightarrow |p(z)| < |p(\zeta)|$.   Since, for $j>\deg(q)/\deg(p)$, the integrals are zero (by taking $R$ to $\infty$), we conclude that $\sigma_{q|p}=\hat{q}\circ p$ with $\deg(\hat{q})\le\deg(q)/\deg(p)$.
\end{proof}

\begin{proof} [Proof of Theorem \ref{T4.1}]
Let $Q$ be a monic polynomial of degree $nk$. By Lemma~\ref{L4.2}, $\sigma_{Q|P}(z)=\hat{Q}\circ P$ where $\deg(\hat{Q})\le n$. In fact, since $P$ is monic of degree $k$ and $Q$ is monic of degree $nk$ it follows from \eqref{4.3} that $\hat{Q}$ is monic of degree $n$. In addition, it follows from the definition of the average that $\norm{\sigma_{Q|P}}_{\fre_P}\le \norm{Q}_{\fre_P}$. Thus, $\norm{T_n\circ P}_{\fre_P} = \norm{T_n}_\fre \le \norm{\hat{Q}}_\fre = \norm{\sigma_{Q|P}}_{\fre_P} \le \norm{Q}_{\fre_P}$ so $T_n\circ P$ is the $nk$-th Chebyshev polynomial of $\fre_p$.

To get the equality of Widom factors note that $\norm{T_n}_\fre=\norm{T_n\circ P}_{\fre_P}$ and $C(\fre_P)^k=C(\fre)$ by Lemma~\ref{L4.1A}.
\end{proof}


\end{document}